\documentclass[12pt]{amsart}
\usepackage{amsmath,amssymb,amsbsy,amsfonts,latexsym,amsopn,amstext,cite,
                                               amsxtra,euscript,amscd,bm,mathabx}
\usepackage{url}
\usepackage[colorlinks,linkcolor=blue,anchorcolor=blue,citecolor=blue,backref=page]{hyperref}
\usepackage{color}
\usepackage{graphics,epsfig}
\usepackage{graphicx}
\usepackage{float} 
\usepackage[english]{babel}
\usepackage{mathtools}
\usepackage{todonotes}
\usepackage{url}
\usepackage[colorlinks,linkcolor=blue,anchorcolor=blue,citecolor=blue,backref=page]{hyperref}
\DeclareMathAlphabet{\mathmybb}{U}{bbold}{m}{n}

\usepackage[norefs,nocites]{refcheck}
\hypersetup{breaklinks=true}

\usepackage[norefs,nocites]{refcheck}
\usepackage[english]{babel}
\begin{document}

\newtheorem{thm}{Theorem}
\newtheorem{lem}[thm]{Lemma}
\newtheorem{claim}[thm]{Claim}
\newtheorem{cor}[thm]{Corollary}
\newtheorem{prop}[thm]{Proposition} 
\newtheorem{definition}[thm]{Definition}
\newtheorem{rem}[thm]{Remark} 
\newtheorem{question}[thm]{Question}
\newtheorem{conj}[thm]{Conjecture}
\newtheorem{prob}{Problem}
\newtheorem{Process}[thm]{Process}
\newtheorem{Computation}[thm]{Computation}
\newtheorem{Fact}[thm]{Fact}
\newtheorem{Observation}[thm]{Observation}

\newtheorem{lemma}[thm]{Lemma}

\newcommand{\GL}{\operatorname{GL}}
\newcommand{\SL}{\operatorname{SL}}
\newcommand{\lcm}{\operatorname{lcm}}
\newcommand{\ord}{\operatorname{ord}}
\newcommand{\Op}{\operatorname{Op}}
\newcommand{\Tr}{\operatorname{Tr}}
\newcommand{\Nm}{\operatorname{Nm}}
\newcommand{\BigSquare}[1]{\raisebox{-0.5ex}{\scalebox{2}{$\square$}}_{#1}}
\newcommand{\legendre}[2]{\left(\frac{#1}{#2}\right)}

\numberwithin{equation}{section}
\numberwithin{thm}{section}
\numberwithin{table}{section}

\numberwithin{figure}{section}

\def\sssum{\mathop{\sum\!\sum\!\sum}}
\def\ssum{\mathop{\sum\ldots \sum}}
\def\iint{\mathop{\int\ldots \int}}

\def\wt {\mathrm{wt}}
\def\Tr {\mathrm{Tr}}

\def\SrA{\cS_r\(\cA\)}

\def\vol {{\mathrm{vol\,}}}
\def\squareforqed{\hbox{\rlap{$\sqcap$}$\sqcup$}}
\def\qed{\ifmmode\squareforqed\else{\unskip\nobreak\hfil
\penalty50\hskip1em\null\nobreak\hfil\squareforqed
\parfillskip=0pt\finalhyphendemerits=0\endgraf}\fi}

\def \ss{\mathsf{s}} 

\def \balpha{\bm{\alpha}}
\def \bbeta{\bm{\beta}}
\def \bgamma{\bm{\gamma}}
\def \blambda{\bm{\lambda}}
\def \bchi{\bm{\chi}}
\def \bphi{\bm{\varphi}}
\def \bpsi{\bm{\psi}}
\def \bomega{\bm{\omega}}
\def \btheta{\bm{\vartheta}}

\newcommand{\bfxi}{{\boldsymbol{\xi}}}
\newcommand{\bfrho}{{\boldsymbol{\rho}}}

 \def \xbar{\overline x}
  \def \ybar{\overline y}

\def\cA{{\mathcal A}}
\def\cB{{\mathcal B}}
\def\cC{{\mathcal C}}
\def\cD{{\mathcal D}}
\def\cE{{\mathcal E}}
\def\cF{{\mathcal F}}
\def\cG{{\mathcal G}}
\def\cH{{\mathcal H}}
\def\cI{{\mathcal I}}
\def\cJ{{\mathcal J}}
\def\cK{{\mathcal K}}
\def\cL{{\mathcal L}}
\def\cM{{\mathcal M}}
\def\cN{{\mathcal N}}
\def\cO{{\mathcal O}}
\def\cP{{\mathcal P}}
\def\cQ{{\mathcal Q}}
\def\cR{{\mathcal R}}
\def\cS{{\mathcal S}}
\def\cT{{\mathcal T}}
\def\cU{{\mathcal U}}
\def\cV{{\mathcal V}}
\def\cW{{\mathcal W}}
\def\cX{{\mathcal X}}
\def\cY{{\mathcal Y}}
\def\cZ{{\mathcal Z}}
\def\Ker{{\mathrm{Ker}}}

\def\NmQR{N(m;Q,R)}
\def\VmQR{\cV(m;Q,R)}

\def\Xm{\cX_{p,m}}

\def \A {{\mathbb A}}
\def \B {{\mathbb A}}
\def \C {{\mathbb C}}
\def \F {{\mathbb F}}
\def \G {{\mathbb G}}
\def \L {{\mathbb L}}
\def \K {{\mathbb K}}
\def \N {{\mathbb N}}
\def \PP {{\mathbb P}}
\def \Q {{\mathbb Q}}
\def \R {{\mathbb R}}
\def \Z {{\mathbb Z}}
\def \fS{\mathfrak S}
\def \fB{\mathfrak B}

\def\Fq{\F_q}
\def\Fqr{\F_{q^r}} 
\def\ovFq{\overline{\F_q}}
\def\ovFp{\overline{\F_p}}
\def\GL{\operatorname{GL}}
\def\SL{\operatorname{SL}}
\def\PGL{\operatorname{PGL}}
\def\PSL{\operatorname{PSL}}
\def\li{\operatorname{li}}
\def\sym{\operatorname{sym}}

\def\Mob{M{\"o}bius }

\def\fF{\EuScript{F}}
\def\M{\mathsf {M}}
\def\T{\mathsf {T}}

\def\e{{\mathbf{\,e}}}
\def\ep{{\mathbf{\,e}}_p}
\def\eq{{\mathbf{\,e}}_q}

\def\\{\cr}
\def\({\left(}
\def\){\right)}

\def\<{\left(\!\!\left(}
\def\>{\right)\!\!\right)}
\def\fl#1{\left\lfloor#1\right\rfloor}
\def\rf#1{\left\lceil#1\right\rceil}

\def\Tr{{\mathrm{Tr}}}
\def\Nm{{\mathrm{Nm}}}
\def\Im{{\mathrm{Im}}}

\def \oF {\overline \F}

\newcommand{\pfrac}[2]{{\left(\frac{#1}{#2}\right)}}

\def \Prob{{\mathrm {}}}
\def\e{\mathbf{e}}
\def\ep{{\mathbf{\,e}}_p}
\def\epp{{\mathbf{\,e}}_{p^2}}
\def\em{{\mathbf{\,e}}_m}

\def\Res{\mathrm{Res}}
\def\Orb{\mathrm{Orb}}

\def\vec#1{\mathbf{#1}}
\def \va{\vec{a}}
\def \vb{\vec{b}}
\def \vh{\vec{h}}
\def \vk{\vec{k}}
\def \vs{\vec{s}}
\def \vu{\vec{u}}
\def \vv{\vec{v}}
\def \vz{\vec{z}}
\def\flp#1{{\left\langle#1\right\rangle}_p}
\def\T {\mathsf {T}}

\def\sfG {\mathsf {G}}
\def\sfK {\mathsf {K}}

\def\mand{\qquad\mbox{and}\qquad}

\title[Cubic polynomials and sums of two squares]
{Cubic polynomials and sums of two squares}

\author[Siddharth Iyer] {Siddharth Iyer}
\address{School of Mathematics and Statistics, University of New South Wales, Sydney, NSW 2052, Australia}
\email{siddharth.iyer@unsw.edu.au}

\begin{abstract}
We establish a lower bound for the frequency with which an irreducible monic cubic polynomial with negative discriminant can be expressed as a sum of two squares ($\BigSquare{2}$). This provides a quantitative answer to a question posed by Grechuk (2021) concerning the infinitude of such values. Our proof relies on a two-dimensional unit argument and the arithmetic of degree six number fields. For example, we show that if $h \equiv 2 \pmod{4}$, then
\begin{align*}
\# \{n : n^3+h \in \BigSquare{2}, \ 1 \leq n \leq x \} \gg x^{1/3-o(1)}.
\end{align*}
These arguments may be generalised to study the representation of irreducible monic cubic polynomials by the quadratic form $x^2+ny^2$, where $n \in \mathbb{N}$.
\end{abstract}

\keywords{Sums of two squares, Cubic Polynomials, Ch\^{a}telet surface, Diophantine Equation, Units}
\subjclass[2020]{11D25, 11D45, 11R09, 11R16, 11R27}

\maketitle

\tableofcontents

\section{Introduction}
In this paper, we aim to study integer points on the surface $y^2+z^2 = \mathbf{p}(x)$, where $\mathbf{p}(x)$ denotes an irreducible monic polynomial in $\Z[x]$ of degree-three.
Let $B_{\mathbf{p}}(X)$ be the set of integers $n$ with $n \leq X$ such that $\mathbf{p}(n)$ is a sum of two squares ($\mathbf{p}(n)\in \BigSquare{2}$). We prove the following theorem.
\begin{thm}
\label{cubictheorem}
If $\mathbf{p}(x) = x^3+a_{2}x^2+a_{1}x+a_{0} \in \Z[x]$ is an irreducible cubic polynomial which satisfies either condition 1) or condition 2) where
\begin{enumerate}
\item $a_{2}^2-a_{1}$ is even, and $a_{1}a_{2}-a_{0}$ is odd
\item $a_{0} \equiv 2 \mod 4$, and $a_{1},a_{2} \equiv 0 \mod 4$,
\end{enumerate}
and $\mathbf{p}$ has one real root and two complex roots then
\begin{align*}
\#B_{\mathbf{p}}(X) \gg X^{1/3 - o(1)}.
\end{align*}
\end{thm}
Grechuk \cite{Grechuk}, \cite[p. 406]{Grechbook} asked if there are infinitely many integers $n$ such that $n^3-2$ is a sum of two squares. Here we answer his question in a quantitative manner, showing that $B_{x^3-2}(X)\gg X^{1/3-o(1)}$. He also asked if there are non-trivial polynomial substitutions $R(t), S_{1}(t)$ and $S_{2}(t)$ such that $R(t)^3-2 = S_{1}(t)^2+S_{2}(t)^2$. Such a substitution is indeed possible if we take
{\tiny
\begin{align}
\label{Substitutionforncubedmin2}
R(t) &= 1000 t^6 + 1800 t^5 + 1380 t^4 - 24 t^3 - 402 t^2 - 130 t + 79, \\
S_{1}(t) &= -18000t^9 - 56400t^8 - 77040t^7 - 44488t^6 + 3312t^5 + 18312t^4 + 5264t^3 - 2466t^2 - 1287t + 346, \notag \\
S_{2}(t) &= 26000t^9 + 64800t^8 + 71280t^7 + 21816t^6 - 20784t^5 - 21504t^4 - 1360t^3 + 4146t^2 + 1263t - 611. \notag
\end{align}
}
We remark that the above equations were derived using system $\mathbf{T}$, described in Section \ref{Substitutioninspirationcubicpol}, and we will briefly explain how they were gathered in Section \ref{Remarkcubicpolys}. Note that the arguments presented later for the proof of Theorem \ref{cubictheorem} cannot be straightforwardly generalised to the case of irreducible cubic polynomials with three distinct real roots.

Our result builds upon the work of Friedlander and Iwaniec regarding values of quadratic polynomials expressible as values of a quadratic form \cite{Friedlander}. They prove that if $g(n) = an^2+bn+c$ is a polynomial with integer coefficients, and $\phi$ is a binary quadratic form, such that $g(n)$ and $\phi$ are ``compatible,'' then the number of $n \leq X$ such that $g(n)$ is represented by $\phi$ is $\gg \frac{X}{(\log X)^{1/2}}$ \cite[Theorem 2]{Friedlander}. We further remark that their result does not apply when $g$ is a reducible polynomial: if we take $\phi(y,z) = y^2+z^2$ and $g(n) = n(n+1)$, the number of $n \leq X$ such that $g(n)$ is a sum of two squares is $\ll \frac{X}{\log(X)}$ (this result is derived by applying the Brun or Selberg sieve, and has been previously stated by Hooley \cite[p. 208]{Hooley}).

Because $4n^3+3$ cannot be expressed as a sum of two squares, extending Theorem \ref{cubictheorem} to non-monic cubic polynomials demands further arithmetic consideration. Finally, we remark that we do not investigate the frequency of representation of cubic polynomials by other quadratic forms; our lower bound arguments should generalize to all quadratic forms of negative discriminant and monic polynomials of negative discriminant, but not for positive discriminants.

This paper is also related to a result of Iwaniec and Munshi \cite{Munshi} regarding densities of rational points on the surface $\phi(y,z) = \mathbf{p}(x)$, where $\phi$ is an irreducible quadratic form with rational coefficients, and $\mathbf{p}(x)$ is an irreducible cubic polynomial in $\Q[x]$. We note additional related work regarding the study of Ch\^{a}telet surfaces \cite{Breteche}, \cite{Irving}, \cite{Woopaper}.

We remark that the bound in Theorem \ref{cubictheorem} is not always optimal. If we can find an integer-valued quadratic polynomial $q_{0}$, and polynomials $q_{1},q_{2} \in \Q[x]$ such that $\mathbf{p}(q_{0}(x)) = q_{1}(x)^2 +q_{2}(x)^2$, then $\#B_{\mathbf{p}}(X) \gg ~ X^{1/2}$. Through trial and error, we obtain the identity
\begin{align*}
(x^2+8)^3+17 = (x^3+10x)^2+(2x^2+23)^2,
\end{align*}
so that $\#B_{x^3+17}(X) \gg X^{1/2}$. Also, when $n$ is an integer, we trivially have $\#B_{x^3+n^2}(X) \gg X^{1/2}$, as for any integer $m$, the quantity $m^6+n^2$ is a sum of two squares. In general, we have not been able to find such polynomial substitutions for an arbitrary monic polynomial in $\Z[x]$ of degree-three. We pose the following question. 
\begin{question}
Is it true that for any irreducible monic polynomial in $\Z[x]$ of degree-three, say $\mathbf{p}(x)$, there exists an integer-valued quadratic polynomial $q_{0}(x)$, and polynomials $q_{1}(x),q_{2}(x) \in \Q[x]$ such that $\mathbf{p}(q_{0}(x)) = q_{1}(x)^2 +q_{2}(x)^2$?
\end{question}
Related to this conjecture, a result by Schinzel in \cite[Lemma 10]{Schinzel} demonstrates that for any such polynomial $\mathbf{p}(x)$, there exists a quadratic polynomial $\mathbf{h}(x) \in \Z[x]$ and a cubic polynomial $\mathbf{g}(x) \in \Z[x]$ such that $\mathbf{g}(x)$ divides $\mathbf{p}(\mathbf{h}(x))$. Perhaps applying the techniques used in \cite{Schinzel} to the polynomial ring $\Z[i][x]$ may settle the conjecture.

At a high level, the proof of Theorem \ref{cubictheorem} hinges on constructing a ``sufficiently dense'' family of degree-six integer-valued polynomials $R(\cdot)$ with sparsely intersecting value sets, such that $\mathbf{p}(R(t))$ is a sum of two squares for every integer $t$.

Let $\mathbf{p}(x)$ be an irreducible polynomial in $\Z[x]$ with $\mathbf{p}(\theta) = 0$ for some $\theta \in \C$. To study when $\mathbf{p}(n) \in \BigSquare{2}$, we investigate whether $q(n-\theta)$ is a sum of two squares in the ring $\Z[\theta]$ when $q \in \BigSquare{2}\setminus\{0\}$. Define
\begin{align*}
\BigSquare{2}(\theta) = \{w_{1}^2+w_{2}^2: \ w_{i}\in \Z[\theta]\}.
\end{align*}
We have the following lemma:
\begin{lem}
\label{Ringtransferprinclemsum2squares}
If an integer $z \in \Z$ satisfies 
\begin{align*}
q(z-\theta) \in \BigSquare{2}(\theta)
\end{align*}
for some $q \in \BigSquare{2}\setminus\{0\}$, then $\mathbf{p}(z) \in \BigSquare{2}$.
\end{lem}
Within the ring $\Z[\theta]$, we define a notion of maximum coefficient, where $M_{\theta}(c_{0}+c_{1}\theta+c_{2}\theta^2) = \max\{|c_{0}|,|c_{1}|,|c_{2}|\}$. For $Y,Y_{1},Y_{2} \geq 0$, define
\begin{align*}
\BigSquare{2}(\theta;Y) :=\{w_{1}^2+w_{2}^2: \ w_{i}\in \Z[\theta], \ M_{\theta}(w_{i}) \leq Y\},
\end{align*}
and for $q \in \BigSquare{2}\setminus\{0\}$ ($q$ will be a parameter in $\BigSquare{2}\setminus\{0\}$ for the rest of this paper) set
\begin{align*}
\mathcal{E}_{\theta,q}(Y_{1},Y_{2}) := \{n: \ n\in \Z, \ q(n-\theta) \in \BigSquare{2}(\theta;Y_{1}), \ |n| \leq Y_{2}\}.
\end{align*}
With Lemma \ref{Ringtransferprinclemsum2squares} we have
\begin{align}
\label{Tacticalringinequalitysum2squares}
\#B_{\mathbf{p}}(X) \geq \# \mathcal{E}_{\theta,q}(Y,X)
\end{align}
for all $X,Y \geq 0$ and $q \in \BigSquare{2}$. Let
\begin{align*}
\mathcal{R}_{\theta,q}(Y_{1},Y_{2}) := \left\{(\omega_{1},\omega_{2}): \  
\begin{aligned} 
&\omega_{i}\in \Z[\theta], \ M_{\theta}(\omega_{i}) \leq Y_{1}, \\ 
&\exists n \in \Z, \ |n| \leq Y_{2}, \ \omega_{1}^2+\omega_{2}^2 = q(n-\theta)
\end{aligned}\right\}.
\end{align*}
We relate the sets $\mathcal{R}_{\theta,q}(Y_{1},Y_{2})$ and $\mathcal{E}_{\theta,q}(Y_{1},Y_{2})$, provided $\theta$ is ``3-valid", that is, if $\theta$ is a real algebraic integer of degree-three and has two complex conjugates.
\begin{lem}
\label{SmallCollisionlemmasum2squares}
If $\theta$ is 3-valid, $\varepsilon>0$ and $Y_{1},Y_{2} \geq 2$ we have
\begin{align*}
\#\mathcal{E}_{\theta,q}(Y_{1},Y_{2})\gg \#\mathcal{R}_{\theta,q}(Y_{1},Y_{2})\cdot Y_{2}^{-\varepsilon}\cdot \log^{-2}Y_{1},
\end{align*}
where the implied constant depends only on $q,\varepsilon$ and $\theta$.
\end{lem}
Lower bounds for $\#\mathcal{R}_{\theta,q}(\cdot,\cdot),$ alongside Lemma \ref{SmallCollisionlemmasum2squares} and inequality \eqref{Tacticalringinequalitysum2squares}, can be used to obtain lower bounds for $\#B_{\mathbf{p}}(\cdot)$. In the subsequent lemmas we pursue lower-bounds for $\#\mathcal{R}_{\theta,1}(\cdot,\cdot)$ under parity conditions on the coefficients of $\mathbf{p}$.
\begin{lem}
\label{Rtheta1example}
If $\mathbf{p}(x) = x^3+a_{2}x^2+a_{1}x+a_{0} \in \Z[x]$ is an irreducible cubic polynomial with $a_{2}^2-a_{1}$ even, and $a_{1}a_{2}-a_{0}$ odd, and $\theta$ is a root of $\mathbf{p}$, we have the estimate
\begin{align*}
\#\mathcal{R}_{\theta,1}(X,X) \gg X^{1/3-o(1)}.
\end{align*}
\end{lem}
\begin{lem}
\label{Rtheta2example}
If $\mathbf{p}(x) = x^3+a_{2}x^2+a_{1}x+a_{0} \in \Z[x]$ is an irreducible cubic polynomial with $a_{1},a_{2} \equiv 0 \mod 4$, and $a_{0} \equiv 2 \mod 4$, and $\theta$ is a root of $\mathbf{p}$, we have the estimate
\begin{align*}
\#\mathcal{R}_{\theta,2}(X,X) \gg X^{1/3-o(1)}.
\end{align*}
\end{lem}
Thus Theorem \ref{cubictheorem} follows from Lemma \ref{Rtheta1example}, Lemma \ref{Rtheta2example}, Lemma \ref{SmallCollisionlemmasum2squares} and inequality \eqref{Tacticalringinequalitysum2squares}.

\section{Proof of Lemma \ref{Ringtransferprinclemsum2squares}}
Note that $\mathbf{p}(x)$ is irreducible over $\Q[i]$. If $q(z-\theta) \in \BigSquare{2}(\theta)$, then we have
\begin{align}
\label{z-alphacomplex}
&q(z-\theta) = (v_{0}+v_{1}\theta+v_{2}\theta^2)^2 + (w_{0}+w_{1}\theta+w_{2}\theta^2)^2 \notag\\
&=(v_{0}+v_{1}\theta+v_{2}\theta^2 +i(w_{0}+w_{1}\theta+w_{2}\theta^2))\times\\ &\times(v_{0}+v_{1}\theta+v_{2}\theta^2 -i(w_{0}+w_{1}\theta+w_{2}\theta^2))\notag,
\end{align}
for some rational values $v_{d},w_{d}$. Furthermore, whenever $\alpha$ is a root of $\mathbf{p}$, the identity \eqref{z-alphacomplex} holds with $\theta$ replaced with $\alpha$ (due to the isomorphism $\Q[i][\theta] \rightarrow \Q[i][\alpha]$ fixing $\Q[i]$ and sending $\theta$ to $\alpha$). In particular, we have that
\begin{align*}
P_{1} := \prod_{\theta: \ \mathbf{p}(\theta) = 0}(v_{0}+v_{1}\theta+v_{2}\theta^2 +i(w_{0}+w_{1}\theta+w_{2}\theta^2))
\end{align*}
is a symmetric polynomial over $\theta_{1},\theta_{2},\theta_{3}$ in $\Q[i][\theta_{1},\theta_{2},\theta_{3}]$, where $\theta_{1},\theta_{2},\theta_{3}$ are roots of $\mathbf{p}(x)$ (here we used that $\mathbf{p}(x)$ is irreducible over $\Q[i]$). Thus, $P_{1}$ is a Gaussian rational, say $u+vi \in \Q[i]$. Also note that
\begin{align*}
P_{2} &:= \prod_{\theta: \ \mathbf{p}(\theta) = 0}(v_{0}+v_{1}\theta+v_{2}\theta^2 -i(w_{0}+w_{1}\theta+w_{2}\theta^2))\\
&= \overline{P_{1}} = u-vi.
\end{align*}
Now, with \eqref{z-alphacomplex}, we write
\begin{align*}
q^3\cdot\mathbf{p}(z) = \prod_{\theta: \ \mathbf{p}(\theta) = 0}q(z-\theta) = P_{1}P_{2} = u^2+v^2.
\end{align*}
Thus, $q^3\cdot\mathbf{p}(z)$ is an integer and a sum of two rational squares, implying that $q^3\cdot\mathbf{p}(z) \in \BigSquare{2}$. Since $q \in \BigSquare{2}\setminus\{0\},$ we must have that $\mathbf{p}(z)$ is a sum of two rational squares. Since $\mathbf{p}(z)$ is an integer, we gather that $\mathbf{p}(z) \in \BigSquare{2}$. 
\section{Proof of Lemma \ref{SmallCollisionlemmasum2squares}}
For $Y,Y_{1},Y_{2} \geq 0$ and $q \in \mathbb{N}$, define
\begin{align*}
\mathcal{A}_{\theta}(\omega,Y) := \left\{(\omega_{1},\omega_{2}): \ \omega_{1}^2+\omega_{2}^2 = \omega, \ M_{\theta}(\omega_{i}) \leq Y, \ \omega_{i} \in \Z[\theta]\right\},
\end{align*}
and
\begin{align*}
\mathcal{C}_{\theta}(Y) := \left\{(c_{0}+id_{0}) + (c_{1}+id_{1})\theta +(c_{2}+id_{2})\theta^2: \ c_{j},d_{j}\in \Z, \ |c_{j}|,|d_{j}| \leq Y\right\}.
\end{align*}
Observe that we have
\begin{align}
\label{Ethetaqsetequation}
\mathcal{E}_{\theta,q}(Y_{1},Y_{2}) = \left\{n: \ n\in \Z, \  \#\mathcal{A}_{\theta}(q(n-\theta),Y_{1}) > 0, \ |n| \leq Y_{2}\right\}.
\end{align}
and let $\mathbf{W}_{\theta}: \Q[\theta]\times \Q[\theta] \rightarrow \Q[\theta,i]$ be the bijective map
\begin{align*}
\mathbf{W}_{\theta}(\omega_{1},\omega_{2}) = \omega_{1} + i\omega_{2}.
\end{align*}
When $\theta_{1}$ and $\theta_{2}$ are roots of a degree-$3$ monic irreducible polynomial, say $\mathbf{p}(x)$ in $\Z[x]$, we observe that there exists a unique field isomorphism $\kappa_{\theta_{1},\theta_{2}} : \Q[\theta_{1},i] \rightarrow \Q[\theta_{2},i]$ which fixes $\Q[i]$ but sends $\theta_{1}$ to $\theta_{2}$ (as the polynomial $\mathbf{p}(x)$ is irreducible over $\Q[i]$).
\begin{lem}
\label{Ethethaqlowerbound1}
Let 
\begin{align*}
\mathbf{m}(\theta,q,Y_{1},Y_{2}) = \max\limits_{n \in \mathbb{Z}, \ \lvert n \rvert \leq Y_2}\# \mathcal{A}_{\theta}(q(n-\theta),Y_{1}) + 1.
\end{align*}
We have
\begin{align*}
\#\mathcal{E}_{\theta,q}(Y_{1},Y_{2}) \geq \frac{\#\mathcal{R}_{\theta,q}(Y_{1},Y_{2})}{\mathbf{m}(\theta,q,Y_{1},Y_{2})}.
\end{align*}
\end{lem}
\begin{proof}
Observe that
\begin{align*}
&\bigcup_{n\in\Z, \ |n| \leq Y_{2}}\mathcal{A}_{\theta}(q(n-\theta),Y_{1}) = \\
&= \left\{(\omega_{1},\omega_{2}): \  
\begin{aligned} 
&\omega_{i}\in \Z[\theta], \ M_{\theta}(\omega_{i}) \leq Y_{1}, \\ 
&\exists n \in \Z, \ |n| \leq Y_{2}, \ \omega_{1}^2+\omega_{2}^2 = q(n-\theta)
\end{aligned}\right\}\\
&= \mathcal{R}_{\theta,q}(Y_{1},Y_{2}).
\end{align*}
Furthermore, for $n_{1} \neq n_{2}$ we have
\begin{align*}
\mathcal{A}_{\theta}(q(n_{1}-\theta),Y_{1})\cap \mathcal{A}_{\theta}(q(n_{2}-\theta),Y_{1}) = \emptyset.
\end{align*}
Hence with \eqref{Ethetaqsetequation}, we have
\begin{align*}
\# \mathcal{R}_{\theta,q}(Y_{1},Y_{2}) &= \sum_{n \in \Z, \ |n| \leq Y_{2}}\#\mathcal{A}_{\theta}(q(n-\theta),Y_{1})\\
&= \sum_{n \in \mathcal{E}_{\theta,q}(Y_{1},Y_{2})}\#\mathcal{A}_{\theta}(q(n-\theta),Y_{1})\\
&\leq \mathbf{m}(\theta,q,Y_{1},Y_{2})\cdot\#\mathcal{E}_{\theta,q}(Y_{1},Y_{2}).
\end{align*}
\par \vspace{-\baselineskip} \qedhere
\end{proof}
\begin{lem}
\label{LowerUpperboundsCthetaY}
If $Y \geq 1$ and $w \in \mathcal{C}_{\theta}(Y)$ is a non-zero element of $\Q[\theta,i]$ then 
\begin{align*}
Y^{-2}\ll |w| \ll Y,
\end{align*}
where $|\cdot|$ denotes the complex modulus operator.
\end{lem}
\begin{proof}
The upper bound follows from the triangle inequality. Let $\mathbf{p}(x)$ be the degree-three monic irreducible polynomial for $\theta$ with roots $\theta_{1},\theta_{2}$ and $\theta_{3}$ respectively, where without loss of generality we may assume $\theta = \theta_{1}$. Suppose that 
\begin{align*}
w = (c_{0}+id_{0}) + (c_{1}+id_{1})\theta +(c_{2}+id_{2})\theta^2
\end{align*}
for integers $c_{j},d_{j}$. Then 
\begin{align*}
\prod_{j=1}^{3}\left((c_{0}+id_{0}) + (c_{1}+id_{1})\theta_{j} +(c_{2}+id_{2})\theta_{j}^2\right)
\end{align*}
is a symmetric polynomial in $\Z[i][\theta_{1},\theta_{2},\theta_{3}]$ and thus is a Gaussian integer, necessarily non-zero as $w \neq 0$. Thus we have
\begin{align*}
|w|\left|\prod_{j=2}^{3}\left((c_{0}+id_{0}) + (c_{1}+id_{1})\theta_{j} +(c_{2}+id_{2})\theta_{j}^2\right)\right| \geq 1,
\end{align*}
or $|w|\cdot Y^{2} \gg 1$.
\end{proof}
\begin{lem}
\label{betaimnotinZ*}
If $\beta$ is an algebraic integer of degree-three, then $\Im(\beta) \not \in \Z\setminus\{0\}$.
\end{lem}
\begin{proof}
Suppose $\beta = \mu + ih$ is an algebraic integer of degree-three for some $h \in \Z \setminus\{0\}$ and $\mu \in \R$. Then it has a minimal polynomial $x^3+a_{2}x^2+a_{1}x+a_{0}\in \Z[x]$ with roots $\mu + ih$, $\mu -ih,$ and $\theta$, where $\theta$ is real. By comparing with values of coefficients, we have $\theta + 2\mu \in \Z$ and $2\theta\mu+\mu^2+h^2 \in \Z$; therefore, $\theta$ is a root of a quadratic polynomial in $\Q[x]$, but this is impossible as $\theta$ is of degree-three (otherwise $\beta$ would be of degree $\leq 2$).  
\end{proof}
\begin{lem}
\label{theta+ialgdeg6}
If $\theta$ is an algebraic integer of degree-three with minimal polynomial $P(x)$ then $\theta+i$ is an algebraic integer of degree-six. Furthermore the minimal polynomial of $\theta+i$ over $\Z$ is $Q(x)$, where
\begin{align*}
Q(x) = \prod_{d=1}^{3}(x-(\theta_{d}+i))(x-(\theta_{d}-i)),
\end{align*}
here $\theta_{1},\theta_{2}$ and $\theta_{3}$ are the distinct roots of $P(x)$.
\end{lem}
\begin{proof}
Let $P(x) = x^3+a_{2}x^2+a_{1}x+a_{0} \in \Z[x]$. We can consider the polynomials $P_{1}(x)$ and $P_{2}(x)$ in $\Z[i][x]$ which are defined as the binomial expansions of $P(x+i)$ and $P(x-i),$ respectively. It is clear that $P(x)$ is irreducible over $\Q[i]$ (otherwise $[\Q[i]:\Q] \geq 3$, which is impossible); thus $P_{1}(x)$ and $P_{2}(x)$ are irreducible over $\Q[i]$. Neither $P_{1}(x)$ nor $P_{2}(x)$ is a polynomial in $\Z[x]$ as the $x^2$ coefficients of the polynomials are $3i+a_{2}$ and $(-3i+a_{2})$ respectively. Observe that the polynomial $Q(x) = P_{1}(x)P_{2}(x)$ has integer-valued coefficients, thus the minimal polynomial of $\theta+i$ over the integers divides $Q(x)$.

If $Q(x)$ factors as $H_{1}(x)H_{2}(x)$, where $H_{j}(x) \in \mathbb{Q}[x]$ and $\deg(H_{j}(x)) \geq 1$, then it factors similarly as $H_{1}(x)\cdot H_{2}(x)$ in $\Q[i][x]$. Since $\Q[i][x]$ is a unique factorization domain with units $c$, where $c \in \Q[i]\setminus\{0\}$, we have $H_{1}(x) = c P_{\sigma(1)}(x)$ and $H_{2}(x) = (1/c) P_{\sigma(2)}(x) $, where $\sigma : \{1,2\}\rightarrow \{1,2\}$ is a permutation map and $c \in \Q[i]$, $c \neq 0$. If $c$ is real, then $H_{j}(x) \not\in \Z[x]$, which is impossible. If $c$ is complex, then the $x^3$ coefficients of $H_{1}(x), H_{2}(x)$ are complex, which is also impossible. Hence the polynomial $Q(x)$ is irreducible over $\Z$.
\end{proof}
\begin{lem}
\label{ranktwounitlemmadeg6}
If $\theta$ is an algebraic integer of degree-three then the group of units of $\mathcal{O}_{\Q[\theta,i]}$ has rank two. 
\end{lem}
\begin{proof}
If $K=\Q[\theta,i]$, then by Lemma \ref{theta+ialgdeg6} we have $K = \Q[\theta+i]$, with the minimal polynomial of $\theta+i$ being $Q(x)$. By Lemma \ref{betaimnotinZ*}, all the roots of $Q$ are complex-valued. Thus by Dirichlet's unit theorem the rank of the group of units of $\mathcal{O}_{K}$ is $(6/2)-1 = 2$.
\end{proof}
\begin{lem}
\label{unitmodulusoneandnotmodulusone}
If $\theta$ is 3-valid, there exists a finite set $W$ of units in $\mathcal{O}_{\Q[\theta,i]}$, such that for every unit $u \in \mathcal{O}_{\Q[\theta,i]}$ there exist integers $t_{1}$ and $t_{2}$ satisfying
\begin{align*}
u = w\cdot l_{1}^{t_{1}}l_{2}^{t_{2}},
\end{align*}
where $l_{1},l_{2}$ are units in $\mathcal{O}_{\Q[\theta,i]}$ with the complex modulus of $l_{1}$ not equal to $1$ ($|l_{1}| \neq 1$), $|l_{2}| = 1$, and $w$ is some unit belonging to $W$.
\end{lem}
\begin{proof}
By Lemma \ref{ranktwounitlemmadeg6}, there exists two multiplicatively independent units $u_{1}$ and $u_{2}$ in $\mathcal{O}_{\Q[\theta,i]}$ such that every unit $u$ in $\mathcal{O}_{\Q[\theta,i]}$ is of the form
\begin{align*}
u= T \cdot u_{1}^{t_{1}}u_{2}^{t_{2}},
\end{align*}
for some $T \in \text{Tor}(\mathcal{O}_{\Q[\theta,i]})$ and integers $t_{1}$ and $t_{2}$. By Dirichlet's unit theorem there exists a unit $q_{1}$ of infinite order in $\mathcal{O}_{\Q[\theta]}$, which is also a real unit in $\mathcal{O}_{\Q[\theta,i]}$. We necessarily have $|q_{1}| \neq 1$; also, by Dirichlet's unit theorem we have $\text{rank}\left(\mathcal{O}_{\Q[\theta]}^{\times}\right) = 1 + 2/2 - 1 = 1$.

Let $V_{\Q[\theta,i]}$ be the group of units of $\Q[\theta,i]$ which have complex modulus of one. Since $\Q[\theta,i]$ is closed under complex conjugation, by \cite[Theorem 1]{Daileda} and Lemma \ref{ranktwounitlemmadeg6} we have
\begin{align*}
\text{rank}\left(\mathcal{O}_{\Q[\theta]}^{\times}\right)+\text{rank}\left(V_{\Q[\theta,i]}\right) = \text{rank}\left(\mathcal{O}_{\Q[\theta,i]}^{\times}\right) = 2.
\end{align*}
Thus, there exists a unit $q_{2}$ of $\mathcal{O}_{\Q[\theta,i]}$ of infinite order with $|q_{2}| = 1$ (this unit is not a torsion element). We observe that there exist integers $s_{1},s_{2},y_{1},y_{2}$ and two torsion elements $g,h \in \text{Tor}(\mathcal{O}_{\Q[\theta,i]})$ such that
\begin{align*}
q_{1} = g \cdot u_{1}^{s_{1}}u_{2}^{s_{2}},
\end{align*}
and
\begin{align*}
q_{2} = h \cdot u_{1}^{y_{1}}u_{2}^{y_{2}}.
\end{align*}
Take $l_{1} = q_{1}g^{-1}$ and $l_{2} = q_{2}h^{-1}$. Observe that $l_{1}$ and $l_{2}$ are multiplicatively independent, thus the matrix 
\begin{align*}
\begin{bmatrix}
s_{1} & y_{1}\\
s_{2} & y_{2}
\end{bmatrix}
\end{align*}
has non-zero determinant. Let $D = |s_{1}y_{2}-y_{1}s_{2}|$, so that $D \geq 1$, and put
\begin{align*}
W = \{T\cdot u_{1}^{v_{1}}u_{2}^{v_{2}}: \ T\in \text{Tor}(\mathcal{O}_{\Q[\theta,i]}), \ v_{i} \in \{1,\ldots,D\}\}.
\end{align*}
The lemma follows with the choices of $l_{1},l_{2}$ and $W$. Here it is useful to recall that $\#\text{Tor}(\mathcal{O}_{\Q[\theta,i]}) < \infty$, and the torsion points of $\mathcal{O}_{\Q[\theta,i]}$ are precisely the roots of cyclotomic polynomials (which have complex modulus one) within $\Q[\theta,i]$.
\end{proof}
\begin{lem}
\label{modulusoneorbitbound}
If $\theta$ is 3-valid, $l_{2} \in \mathcal{O}_{\Q[\theta,i]}$ is a unit of infinite order with $|l_{2}| = 1$, $b \in \Q[\theta,i]$ is non-zero, and $Y \geq 2$ then 
\begin{align*}
\#\left\{n: n\in \Z, \ b\cdot l_{2}^n \in \mathcal{C}_{\theta}(Y) \right\} \ll \log Y,
\end{align*}
where the implied constant is independent of $b$, but dependent on $l_{2}$.
\end{lem}
\begin{proof}
Let $p(x)$ be the minimal monic polynomial of $\theta$ in $\Z[x]$, which is necessarily irreducible over $\Q[i]$; also let $\theta_{1},\theta_{2},\theta_{3}$ be the roots of $p(x)$ with $\theta = \theta_{1}$. Observe that $l_{2}$ is a root of the polynomial $q(x)$ where
\begin{align*}
q(x) = \prod_{j=1}^{3}\left(x-\kappa_{\theta_{1},\theta_{j}}(l_{2})\right)\left(x-\overline{\kappa_{\theta_{1},\theta_{j}}(l_{2})}\right).
\end{align*}
Observe that if $x$ is real, then $q(x)$ is real; thus, $q(x) \in \R[x]$. Since $p(x)$ is irreducible over $\Q[i]$, the polynomials $\prod_{j=1}^{3}\left(x-\kappa_{\theta_{1},\theta_{j}}(l_{2})\right)$ and $\prod_{j=1}^{3}\left(x-\overline{\kappa_{\theta_{1},\theta_{j}}(l_{2})}\right)$ belong to $\Q[i]$ (as both expressions can be expanded as a symmetric polynomial in $\theta_{1},\theta_{2},\theta_{3}$). Hence we deduce that $q(x) \in \R[x]\cap \left(\Q[i][x]\right)$, implying that $q(x) \in \Q[x]$. By Kronecker's theorem, there exists a choice of $j \in \{2,3\}$ such that $\left|\kappa_{\theta_{1},\theta_{j}}(l_{2})\right| \neq 1$. Now note that we wish to bound from above the cardinality of the set $\mathcal{X}_{\theta}(b,l_{2},Y)$, where
\begin{align*}
\mathcal{X}_{\theta}(b,l_{2},Y) = \{n: n \in \Z, \ b\cdot l_{2}^n \in \mathcal{C}_{\theta}(Y)\}.
\end{align*}
Observe that
\begin{align*}
\mathcal{X}_{\theta}(b,l_{2};Y)  = \mathcal{X}_{\theta_{j}}(\kappa_{\theta_{1},\theta_{j}}(b), \kappa_{\theta_{1},\theta_{j}}(l_{2});Y).
\end{align*}
Let $\rho = |\kappa_{\theta_{1},\theta_{j}}(l_{2})|$ and $c = |\kappa_{\theta_{1},\theta_{j}}(b)|$. By Lemma \ref{LowerUpperboundsCthetaY} there exist constants $A,B > 0$ depending on $\theta$ such that
\begin{align*}
\mathcal{X}_{\theta_{j}}(\kappa_{\theta_{1},\theta_{j}}(b), \kappa_{\theta_{1},\theta_{2}}(l_{2});Y) \subseteq \{n: \ n\in \Z, \ A\cdot Y^{-2} \leq c\cdot \rho^{n} \leq B \cdot Y\}.
\end{align*}
The lemma now follows (here we use $\rho \neq 1$).
\end{proof}
For a non-zero element $w \in \Z[i,\theta]$ let $\mathcal{U}(w,Y)$ denote elements $s$ of $\mathcal{C}_{\theta}(Y)$ where $s = w u$ for some unit $u$ in $\mathcal{O}_{\Q[\theta,i]}$.
\begin{lem}
\label{twodimensionalunitlemma}
If $\theta$ is 3-valid, $w \neq 0$ is an element of $\Z[\theta,i]$ and $Y \geq 2$, we have
\begin{align*}
\#\mathcal{U}(w,Y) \ll \log^2 Y,
\end{align*}
where the implied constant is independent of $w$.
\end{lem}
\begin{proof}
From Lemma \ref{unitmodulusoneandnotmodulusone}, there exists a finite set $W$ and units $l_{1}$ and $l_{2}$ in $\mathcal{O}_{\Q[\theta,i]}$ which satisfy $|l_{1}| \neq 1$ and $|l_{2}| = 1$ such that for every unit $u$ in $\mathcal{O}_{\Q[\theta,i]}$ there exists $g \in W$ and integers $t_{1}$ and $t_{2}$ with
\begin{align*}
u = g \cdot l_{1}^{t_{1}}l_{2}^{t_{2}}.
\end{align*}
We may now partition the set $\mathcal{U}(w,Y)$ into sets $\mathcal{U}(g,t_{1};w,Y)$ where $g \in W$ and $t_{1}$ is an integer; more specifically, this defined subset satisfies 
\begin{align*}
\mathcal{U}(g,t_{1};w,Y) = \{t_{2}: \ t_{2}\in \Z, \ wgl_{1}^{t_{1}}l_{2}^{t_{2}}\in \mathcal{C}_{\theta}(Y)\}.
\end{align*}
By Lemma \ref{LowerUpperboundsCthetaY} there exist positive constants $A$ and $B$ (depending only on $\theta$ and the set $W$) such that when $|w||g||l_{1}|^{t_{1}} > B\cdot Y$ or $|w||g||l_{1}|^{t_{1}} < A\cdot Y^{-2}$, the set $\mathcal{U}(g,t_{1};w,Y)$ is empty. Define the quantities
\begin{align*}
H(g) = \left\lceil\log_{|l_{1}|}\left(\frac{B Y}{|w||g|}\right)\right\rceil
\end{align*}
\begin{align*}
h(g) = \left\lfloor\log_{|l_{1}|}\left( \frac{A Y^{-2}}{|w||g|}\right)\right\rfloor
\end{align*}
(here the logarithm is taken in base $|l_{1}|$). We observe
\begin{align*}
\mathcal{U}(w,Y) = \bigcup_{g\in W}\bigcup_{t_{1} = h(g)}^{H(g)}\mathcal{U}(g,t_{1};w,Y).
\end{align*}
Hence by Lemma \ref{modulusoneorbitbound} we have
\begin{align*}
\#\mathcal{U}(w,Y) &\ll \sum_{g\in W}\left|H(g) -h(g)+1\right|\cdot\log Y\\
&\ll \log^2 Y.
\end{align*}
\par \vspace{-\baselineskip} \qedhere
\end{proof}
\begin{lem}
\label{AthetaomegaYupperbound}
Let $\theta$ be 3-valid, and let $\omega \in \Z[\theta]$. Then for every $\varepsilon>0$ and $Y \geq 2$ we have
\begin{align*}
\#\mathcal{A}_{\theta}(\omega,Y) \ll M(\omega)^{\varepsilon}\cdot \log^2 Y,
\end{align*}
where the implied constant depends only on $\theta$ and $\varepsilon$ (effective only with sufficient knowledge of units in $\mathcal{O}_{\Q[\theta,i]}$).
\end{lem}
\begin{proof}
The bound is trivial for $\omega = 0$; for the rest of the proof, assume that $\omega \neq 0$.
Define the set $d(\omega)$ to contain ideals in $\mathcal{O}_{\Q[\theta,i]}$ that divide the ideal $(\omega)$. Since $\mathcal{O}_{\Q[\theta,i]}$ is a Dedekind domain, we have 
\begin{align*}
\# d(\omega) \ll N_{\Q[\theta,i]/\Q}(\omega)^{\varepsilon} \ll M(\omega)^{6\varepsilon}.
\end{align*}
Also let $p(\omega)$ denote the principal ideals of $d(\omega)$ so that $p(\omega) \subset d(\omega)$. We can now define the map $I_{\omega,Y} : \mathcal{A}_{\theta}(\omega,Y) \rightarrow p(\omega)$ where
\begin{align*}
I_{\omega,Y}(\omega_{1},\omega_{2}) = (\omega_{1}+i\omega_{2}).
\end{align*}
The map $I_{\omega,Y}$ induces an equivalence relation on $\mathcal{A}_{\theta}(\omega,Y)$ where $\mathbf{a}\sim \mathbf{b}$ if $I_{\omega,Y}(\mathbf{a}) = I_{\omega,Y}(\mathbf{b})$. Accordingly, we may partition $\mathcal{A}_{\theta}(\omega,Y)$ into $t$ equivalence classes, where $t \ll M(\omega)^{6\varepsilon}$. Pick elements $(u_{1},v_{1}),\ldots (u_{t},v_{t})$ of $\mathcal{A}_{\theta}(\omega,Y)$ that belong to distinct equivalence classes. Define
\begin{align*}
\mathcal{A}_{\theta}(u_{j},v_{j};\omega,Y) = \{\mathbf{s}: \  \mathbf{s}\in \mathcal{A}_{\theta}(\omega,Y), \ \mathbf{s}\sim (u_{j},v_{j})\}
\end{align*}
so that 
\begin{align}
\label{tpartitionAthetawY}
\bigcup_{j=1}^{t}\mathcal{A}_{\theta}(u_{j},v_{j};\omega,Y) = \mathcal{A}_{\theta}(\omega,Y).
\end{align}
We now have to estimate the cardinality of $\mathcal{A}_{\theta}(u_{j},v_{j};\omega,Y)$. Observe that if $(\alpha,\beta) \in \mathcal{A}_{\theta}(u_{j},v_{j};\omega,Y),$ then $(\alpha + i\beta) = (u_{j}+iv_{j})$ (as ideals in $\mathcal{O}_{\Q[\theta,i]}$) and $\alpha+i\beta \in \mathcal{C}_{\theta}(Y)$. The ideal equality requirement happens if and only if $\alpha + i \beta = (u_{j}+iv_{j})u$ for some unit $u$ in $\mathcal{O}_{\Q[\theta,i]}$. Hence we have
\begin{align*}
\mathbf{W}_{\theta}(\mathcal{A}_{\theta}(u_{j},v_{j};\omega,Y)) = \mathcal{U}_{\theta}(u_{j}+iv_{j},Y).
\end{align*}
By Lemma \ref{twodimensionalunitlemma} we have $\#\mathcal{U}_{\theta}(u_{j}+iv_{j},Y) \ll \log^2 Y$. Thus
\begin{align*}
\# \mathcal{A}_{\theta}(u_{j},v_{j};\omega,Y) = \# \mathbf{W}_{\theta}(\mathcal{A}_{\theta}(u_{j},v_{j};\omega,Y)) \ll \log^2 Y.
\end{align*}
By the set relation \eqref{tpartitionAthetawY} we have
\begin{align*}
\# \mathcal{A}_{\theta}(\omega,Y) &\ll t\cdot \log^2 Y\\
&\ll M(\omega)^{6\varepsilon}\cdot \log^2 Y.
\end{align*}
\par \vspace{-\baselineskip} \qedhere
\end{proof}
We now remark that Lemma \ref{SmallCollisionlemmasum2squares} follows from Lemma \ref{Ethethaqlowerbound1} and Lemma \ref{AthetaomegaYupperbound}.
\section{Proof of Lemma \ref{Rtheta1example}}

Throughout this section, assume that $u_{0},u_{1},u_{2},v_{0},v_{1},v_{2}$ are integer-valued variables. We also abbreviate $(u_{0},u_{1},u_{2},v_{0},v_{1},v_{2})$ as $(\mathbf{u},\mathbf{v})$. Define the following functions:
\begin{align*}
g_{0}(\mathbf{u},\mathbf{v}) = &(u_{0}^2-2a_{0}u_{1}u_{2}+u_{2}^2a_{0}a_{2})+\\
&+(v_{0}^2-2a_{0}v_{1}v_{2}+v_{2}^2a_{0}a_{2}),
\end{align*}
\begin{align*}
g_{1}(\mathbf{u},\mathbf{v}) = &\left(2u_{0}u_{1}-2a_{1}u_{1}u_{2}+u_{2}^2(a_{1}a_{2}-a_{0})\right) +\\
&+\left(2v_{0}v_{1}-2a_{1}v_{1}v_{2}+v_{2}^2(a_{1}a_{2}-a_{0})\right),
\end{align*}
\begin{align*}
g_{2}(\mathbf{u},\mathbf{v}) = &(2u_{0}u_{2}+u_{1}^2 -2a_{2}u_{1}u_{2}+u_{2}^2(a_{2}^2-a_{1}))+\\
&+(2v_{0}v_{2}+v_{1}^2 -2a_{2}v_{1}v_{2}+v_{2}^2(a_{2}^2-a_{1})).
\end{align*}
We have the following lemma:
\begin{lem}
\label{Expansionubfvbfsum2squares}
We have
\begin{align*}
&(u_{0}+u_{1}\theta+u_{2}\theta^2)^2+(v_{0}+v_{1}\theta+v_{2}\theta^2)^2 = \notag \\
&g_{0}(\mathbf{u},\mathbf{v}) + g_{1}(\mathbf{u},\mathbf{v})\theta+g_{2}(\mathbf{u},\mathbf{v})\theta^2.
\end{align*}
\end{lem}
\begin{proof}
We have
\begin{align*}
\theta^3 = -a_{2}\theta^2-a_{1}\theta-a_{0},
\end{align*}
\begin{align*}
\theta^{4} &= -a_{2}\theta^3-a_{1}\theta^2-a_{0}\theta\\
&= -a_{2}(-a_{2}\theta^2-a_{1}\theta-a_{0})-a_{1}\theta^2-a_{0}\theta\\
&=(a_{2}^2-a_{1})\theta^2+(a_{1}a_{2}-a_{0})\theta+a_{0}a_{2}.
\end{align*}
We expand
\begin{align*}
&(u_{0}+u_{1}\theta+u_{2}\theta^2)^2 = u_{0}^2 +2u_{0}u_{1}\theta + 2u_{0}u_{2}\theta^2 +u_{1}^2 \theta^2 + 2u_{1}u_{2}\theta^3+u_{2}^2\theta^4\\
&= (u_{0}^2-2a_{0}u_{1}u_{2})+(2u_{0}u_{1}-2a_{1}u_{1}u_{2})\theta +(2u_{0}u_{2}+u_{1}^2 -2a_{2}u_{1}u_{2})\theta^2+u_{2}^2\theta^4\\
&=(u_{0}^2-2a_{0}u_{1}u_{2})+(2u_{0}u_{1}-2a_{1}u_{1}u_{2})\theta +(2u_{0}u_{2}+u_{1}^2 -2a_{2}u_{1}u_{2})\theta^2+\\
&+u_{2}^2a_{0}a_{2}+u_{2}^2(a_{1}a_{2}-a_{0})\theta +u_{2}^2(a_{2}^2-a_{1})\theta^2\\
&= (u_{0}^2-2a_{0}u_{1}u_{2}+u_{2}^2a_{0}a_{2}) + (2u_{0}u_{1}-2a_{1}u_{1}u_{2}+u_{2}^2(a_{1}a_{2}-a_{0}))\theta+\\
&+(2u_{0}u_{2}+u_{1}^2 -2a_{2}u_{1}u_{2}+u_{2}^2(a_{2}^2-a_{1}))\theta^2.
\end{align*}
We similarly expand $(v_{0}+v_{1}\theta+v_{2}\theta^2)^2$ and verify the lemma.
\end{proof}
Define
\begin{align*}
I(u_{1},u_{2},v_{1},v_{2}) = u_{1}v_{2}-v_{1}u_{2},
\end{align*}
\begin{align*}
h_{1}(u_{1},u_{2},& v_{1},v_{2}) = \\
&-1+2a_{1}u_{1}u_{2}-u_{2}^2(a_{1}a_{2}-a_{0})+2a_{1}v_{1}v_{2}-v_{2}^2(a_{1}a_{2}-a_{0}),
\end{align*}
\begin{align*}
h_{2}(u_{1},u_{2},& v_{1},v_{2}) = \\
& -u_{1}^2+2a_{2}u_{1}u_{2}-u_{2}^2(a_{2}^2-a_{1}) -v_{1}^2+2a_{2}v_{1}v_{2}-v_{2}^2(a_{2}^2-a_{1}),
\end{align*}
\begin{align*}
U_{0}(u_{1},u_{2},v_{1},v_{2}) = v_{2} h_{1}(u_{1},u_{2},v_{1},v_{2})/2 - v_{1} h_{2}(u_{1},u_{2},v_{1},v_{2})/2
\end{align*}
and
\begin{align*}
V_{0}(u_{1},u_{2},v_{1},v_{2}) = -u_{2} h_{1}(u_{1},u_{2},v_{1},v_{2})/2 +u_{1}h_{2}(u_{1},u_{2},v_{1},v_{2})/2.
\end{align*}
Let $\mathbf{S} \subset \Z[\theta]\times\Z[\theta]$ contain tuples $((u_{0}+u_{1}\theta+u_{2}\theta^2), (v_{0}+v_{1}\theta+v_{2}\theta^2))$ such that 
\begin{enumerate}
\item $v_{2}$ is even, $\gcd(u_{2},v_{2}) = 1$
\item $u_{1}$ and $v_{1}$ are odd, with $I(u_{1},u_{2},v_{1},v_{2}) = 1$
\item $u_{0} = U_{0}(u_{1},u_{2},v_{1},v_{2})$ and $v_{0} = V_{0}(u_{1},u_{2},v_{1},v_{2})$.
\end{enumerate}
\begin{lem}
\label{SsubsetRtheta1inftyinfty}
We have $\mathbf{S} \subset \mathcal{R}_{\theta,1}(\infty,\infty)$.
\end{lem}
\begin{proof}
By Lemma \ref{Expansionubfvbfsum2squares} we have 
\begin{align*}
\left((u_{0}+u_{1}\theta+u_{2}\theta^2),(v_{0}+v_{1}\theta+v_{2}\theta^2)\right) \in \mathcal{R}_{\theta,1}(\infty,\infty)
\end{align*}
if and only if $g_{1}(\mathbf{u},\mathbf{v}) = -1$ and $g_{2}(\mathbf{u},\mathbf{v}) = 0$. Solving the two equations gives us
\begin{align*}
2u_{1}u_{0}+2v_{1}v_{0} = -1+2a_{1}u_{1}u_{2}-u_{2}^2(a_{1}a_{2}-a_{0})+2a_{1}v_{1}v_{2}-v_{2}^2(a_{1}a_{2}-a_{0}),
\end{align*}
and
\begin{align*}
2u_{2}u_{0}+2v_{2}v_{0} = -u_{1}^2+2a_{2}u_{1}u_{2}-u_{2}^2(a_{2}^2-a_{1}) -v_{1}^2+2a_{2}v_{1}v_{2}-v_{2}^2(a_{2}^2-a_{1}),
\end{align*}
where the right-hand sides of the above equations are $h_{1}(u_{1},u_{2},v_{1},v_{2})$ and $h_{2}(u_{1},u_{2},v_{1},v_{2})$ respectively. This leads us to the matrix equation
\begin{align*}
\begin{bmatrix}
2u_{1} & 2v_{1}\\
2u_{2} & 2v_{2}
\end{bmatrix}
\times
\begin{bmatrix}
u_{0} \\
v_{0}
\end{bmatrix}
=
\begin{bmatrix}
h_{1}(u_{1},u_{2},v_{1},v_{2})\\
h_{2}(u_{1},u_{2},v_{1},v_{2})
\end{bmatrix}
.
\end{align*}
If $I(u_{1},u_{2},v_{1},v_{2}) \neq 0$, solving the matrix equation gives us
\begin{align*}
\begin{bmatrix}
u_{0}\\
v_{0}
\end{bmatrix}
&= \frac{1}{u_{1}v_{2}-v_{1}u_{2}}
\begin{bmatrix}
v_{2} & -v_{1}\\
-u_{2} & u_{1}
\end{bmatrix}
\times
\begin{bmatrix}
h_{1}(u_{1},u_{2},v_{1},v_{2})/2\\
h_{2}(u_{1},u_{2},v_{1},v_{2})/2
\end{bmatrix}
\\
&= \frac{1}{u_{1}v_{2}-v_{1}u_{2}}
\begin{bmatrix}
v_{2} h_{1}(u_{1},u_{2},v_{1},v_{2})/2 - v_{1} h_{2}(u_{1},u_{2},v_{1},v_{2})/2 \\
-u_{2} h_{1}(u_{1},u_{2},v_{1},v_{2})/2 +u_{1}h_{2}(u_{1},u_{2},v_{1},v_{2})/2
\end{bmatrix}\\
&=\frac{1}{I(u_{1},u_{2},v_{1},v_{2})}
\begin{bmatrix}
U_{0}(u_{1},u_{2},v_{1},v_{2})\\
V_{0}(u_{1},u_{2},v_{1},v_{2}).
\end{bmatrix}
\end{align*}
We thus verify that $\mathbf{S}\subset \mathcal{R}_{\theta,1}(\infty,\infty)$.
\end{proof}
We now more closely examine the set $\mathbf{S}$. From now on, we will denote $\alpha$ and $\beta$ to be coprime integers such that $\alpha$ is even. Note that there exist odd integers $\mathbf{v}(\alpha,\beta)$ and $\mathbf{u}(\alpha,\beta)$ satisfying $\mathbf{u}(\alpha,\beta) \alpha - \mathbf{v}(\alpha,\beta) \beta = 1$, and in particular 
\begin{align*}
(\mathbf{u}(\alpha,\beta)+2\beta t) \alpha - (\mathbf{v}(\alpha,\beta)+2\alpha t) \beta = 1
\end{align*}
for all integers $t$. Observe that $\mathbf{S}$ satisfies a \textit{completion} property. If integer values $v_{2} = \alpha$, $u_{2} = \beta$ and $v_{1}$ are given with $v_{1} \equiv \mathbf{v}(\alpha,\beta) \pmod{2\alpha}$, then we can find a unique set of integer values $u_{0},v_{0},u_{1}$ such that 
\begin{align*}
((u_{0}+u_{1}\theta+u_{2}\theta^2), (v_{0}+v_{1}\theta+v_{2}\theta^2)) \in \mathbf{S}.
\end{align*}
For $c>0$ we define a subset of $\mathbf{S}$ where 
\begin{align*}
&\mathbf{S}_{c}(\alpha,\beta,X) = \\
&\left\{((u_{0}+u_{1}\theta+u_{2}\theta^2), (v_{0}+v_{1}\theta+v_{2}\theta^2)): \ 
\begin{aligned}
&v_{2} = \alpha, \  u_{2}=\beta, \ |v_{1}| \leq c\cdot X^{1/6}, \\
&v_{1}\equiv \mathbf{v}(\alpha,\beta) \pmod{2\alpha}, \ (\mathbf{u},\mathbf{v})\in \mathbf{S}
\end{aligned}
\right\}.
\end{align*}
Furthermore define 
\begin{align*}
\mathbf{V}_{c}(\alpha,\beta,X) = \{v_{1}: \ v_{1} \equiv \mathbf{v}(\alpha,\beta) \pmod{2\alpha}, \ |v_{1}| \leq c\cdot X^{1/6}\},
\end{align*}
where, by the completion property for $\mathbf{S}$ we remark that the sets $\mathbf{V}_{c}(\alpha,\beta,X)$ and $\mathbf{S}_{c}(\alpha,\beta,X)$ are in bijection and in particular we have
\begin{align}
\label{VcScsamecardinalitysum2squares}
\# \mathbf{V}_{c}(\alpha,\beta,X) = \# \mathbf{S}_{c}(\alpha,\beta,X).
\end{align}
\begin{lem}
\label{betaalphaauxlemmasum2squares}
If $c>0$, there exists a positive constant $F(c)$ with $\lim_{c\rightarrow 0^{+}}F(c) = 0$ such that if $X$ is sufficiently large ($X \geq N_{c}$) and 
\begin{align*}
1\leq \beta \leq \alpha \leq c\cdot X^{1/6},
\end{align*}
we have  
\begin{align*}
\mathbf{S}_{c}(\alpha,\beta,X) \subset \mathcal{R}_{\theta,1}(F(c)\cdot X^{1/2},\infty).
\end{align*}
\end{lem}
\begin{proof}
If
\begin{align*}
((u_{0}+u_{1}\theta+u_{2}\theta^2), (v_{0}+v_{1}\theta+v_{2}\theta^2)) \in \mathbf{S}_{c}(\alpha,\beta,X),
\end{align*}
by Lemma \ref{SsubsetRtheta1inftyinfty} we know that 
\begin{align}
\label{Halfdoneu0u1u2v0v1v2proof}
((u_{0}+u_{1}\theta+u_{2}\theta^2), (v_{0}+v_{1}\theta+v_{2}\theta^2)) \in \mathcal{R}_{\theta,1}(\infty,\infty).
\end{align}
It remains to show that $|u_{i}|,|v_{i}| \leq F(c)\cdot X^{1/2}$ for some function $F$ which satisfies $\lim_{c\rightarrow 0^{+}}F(c) = 0$. Note that $I(u_{1},\beta,v_{1},\alpha) = 1$ or 
\begin{align*}
u_{1}\alpha-v_{1}\beta = 1,
\end{align*}
implying that
\begin{align*}
|u_{1}| = \left|\frac{1+v_{1}\beta}{\alpha}\right| \leq 2 |v_{1}|.
\end{align*}
Using
\begin{align}
\label{v1inequalitysum2squares}
|v_{1}|\leq c\cdot X^{1/6},
\end{align}
we obtain
\begin{align}
\label{u1inequalitysum2squares}
|u_{1}| \leq 2c \cdot X^{1/6}.
\end{align}
On the set $\mathbf{S}_{c}(\alpha,\beta,X),$ we thus have
\begin{align*}
&h_{1}(u_{1},u_{2}, v_{1},v_{2}) = h_{1}(u_{1},\beta, v_{1},\alpha)\\
&= -1+2a_{1}u_{1}\beta-\beta^2(a_{1}a_{2}-a_{0})+2a_{1}v_{1}\alpha-\alpha^2(a_{1}a_{2}-a_{0}),
\end{align*}
giving us the estimate
\begin{align*}
\left|h_{1}(u_{1},u_{2}, v_{1},v_{2})\right| \leq 1+2|a_{1}|c^2 X^{1/3} + 2c^2|(a_{1}a_{2}-a_{0})|X^{1/3}+2c^2|a_{1}|X^{1/3},
\end{align*}
so that if $X$ is sufficiently large we have 
\begin{align}
\label{h1u1u2v1v2inequalitysum2squares}
\left|h_{1}(u_{1},u_{2}, v_{1},v_{2})\right| &\leq \left(4|a_{1}|c^2 + 4c^2|(a_{1}a_{2}-a_{0})|+4c^2|a_{1}|\right)X^{1/3}\notag\\
&= H_{1}(c)\cdot X^{1/3},
\end{align}
where $H_{1}(c)>0$ is some constant that approaches $0$ as $c$ approaches $0$.
Similarly we obtain
\begin{align*}
&h_{2}(u_{1},u_{2}, v_{1},v_{2}) = h_{2}(u_{1},\beta, v_{1},\alpha)\\
&= -u_{1}^2+2a_{2}u_{1}\beta-\beta^2(a_{2}^2-a_{1}) -v_{1}^2+2a_{2}\alpha v_{2}-\alpha^2(a_{2}^2-a_{1}),
\end{align*}
so that with \eqref{v1inequalitysum2squares} and \eqref{u1inequalitysum2squares} we have
\begin{align}
\label{h2u1u2v1v2inequalitysum2squares}
|h_{2}(u_{1},u_{2}, v_{1},v_{2})| \leq H_{2}(c)\cdot X^{1/3},
\end{align}
where $H_{2}(c)>0$ is some constant that approaches $0$ as $c$ approaches $0$. Note that 
\begin{align*}
u_{0} &= U_{0}(u_{1},u_{2},v_{1},v_{2}) \\
&= v_{2} h_{1}(u_{1},u_{2},v_{1},v_{2})/2 - v_{1} h_{2}(u_{1},u_{2},v_{1},v_{2})/2.
\end{align*}
Hence
\begin{align*}
|u_{0}| \leq |v_{2} h_{1}(u_{1},u_{2},v_{1},v_{2})| +|v_{1} h_{2}(u_{1},u_{2},v_{1},v_{2})|,
\end{align*}
Using inequalities \eqref{v1inequalitysum2squares}, \eqref{u1inequalitysum2squares}, \eqref{h1u1u2v1v2inequalitysum2squares}, \eqref{h2u1u2v1v2inequalitysum2squares} and substitutions $v_{2} = \alpha$ and $u_{2} = \beta$, we obtain 
\begin{align}
\label{u0inequalitysum2squares}
|u_{0}| \leq \Theta(c)\cdot X^{1/2},
\end{align}
where $\Theta(c)>0$ is some constant that approaches $0$ as $c$ approaches $0$. Similarly we obtain
\begin{align*}
v_{0} &= V_{0}(u_{1},u_{2},v_{1},v_{2})\\
&= -u_{2} h_{1}(u_{1},u_{2},v_{1},v_{2})/2 +u_{1}h_{2}(u_{1},u_{2},v_{1},v_{2})/2.
\end{align*}
Using inequalities \eqref{v1inequalitysum2squares}, \eqref{u1inequalitysum2squares}, \eqref{h1u1u2v1v2inequalitysum2squares}, \eqref{h2u1u2v1v2inequalitysum2squares} and substitutions $v_{2} = \alpha$ and $u_{2} = \beta$ we obtain
\begin{align}
\label{v0inequalitysum2squares}
|v_{0}| \leq \rho(c)\cdot X^{1/2},
\end{align}
where $\rho(c)>0$ is some constant that approaches $0$ as $c$ approaches $0$. With the choice $F(c) = \max\{2c,\rho(c),\Theta(c)\}$, inequalities \eqref{v1inequalitysum2squares}, \eqref{u1inequalitysum2squares}, \eqref{u0inequalitysum2squares}, \eqref{v0inequalitysum2squares}, and set relation \eqref{Halfdoneu0u1u2v0v1v2proof}, and the discussion stated immediately after \eqref{Halfdoneu0u1u2v0v1v2proof}, the proof follows.
\end{proof}
\begin{lem}
\label{LowerboundScalphabetaX}
If $c > 0$ and $2 \leq \alpha \leq (c/4)\cdot X^{1/6}$ then
\begin{align*}
\#\mathbf{S}_{c}(\alpha,\beta,X) \geq \frac{c\cdot X^{1/6}}{4\alpha}.
\end{align*}
\end{lem}
\begin{proof}
By equation \eqref{VcScsamecardinalitysum2squares} we have $\#\mathbf{S}_{c}(\alpha,\beta,X) = \#\mathbf{V}_{c}(\alpha,\beta,X)$, we recall that 
\begin{align*}
\mathbf{V}_{c}(\alpha,\beta,X) =  \{v_{1}: \ v_{1} \equiv \mathbf{v}(\alpha,\beta) \pmod{2\alpha}, \ |v_{1}| \leq c\cdot X^{1/6}\}.
\end{align*}
Hence
\begin{align*}
\#\mathbf{V}_{c}(\alpha,\beta,X) \geq \left\lfloor \frac{c\cdot X^{1/6}}{2\alpha} \right\rfloor \geq \frac{c\cdot X^{1/6}}{4\alpha}.
\end{align*}
\par \vspace{-\baselineskip} \qedhere
\end{proof}
\begin{lem}
\label{Rtheta1dX1/2inftyinclusionlem}
There exists a constant $d>0$ such that for sufficiently large $X$ we have 
\begin{align*}
\mathcal{R}_{\theta,1}(d X^{1/2},\infty) =  \mathcal{R}_{\theta,1}(d X^{1/2},X).
\end{align*}
\end{lem}
\begin{proof}
Let $d>0$ be sufficiently small such that $2(1+|\theta|+|\theta|^2)^2 d^2 < 1$. Observe that if
\begin{align*}
((u_{0}+u_{1}\theta+u_{2}\theta^2), (v_{0}+v_{1}\theta+v_{2}\theta^2)) \in \mathcal{R}_{\theta,1}(d X^{1/2},\infty),
\end{align*}
then 
\begin{align*}
(u_{0}+u_{1}\theta+u_{2}\theta^2)^2 + (v_{0}+v_{1}\theta+v_{2}\theta^2)^2 = n-\theta
\end{align*}
for some integer $n$. Using the triangle inequality, we find that
\begin{align*}
|n| \leq 2d^2(1+|\theta|+|\theta|^2)^2 X + |\theta|,
\end{align*}
which is less than $X$ when $X$ is sufficiently large. Hence for large values of $X$ we have
\begin{align*}
\mathcal{R}_{\theta,1}(d X^{1/2},\infty) \subseteq  \mathcal{R}_{\theta,1}(d X^{1/2},X),
\end{align*}
and the reverse inclusion follows trivially.
\end{proof}
As in Lemma \ref{Rtheta1dX1/2inftyinclusionlem}, choose $d>0$ and a positive real number $N_{d}$ such that for $X \geq N_{d}$ we have
\begin{align*}
\mathcal{R}_{\theta,1}(d X^{1/2},\infty) =  \mathcal{R}_{\theta,1}(d X^{1/2},X).
\end{align*}
As in Lemma \ref{betaalphaauxlemmasum2squares}, choose $c>0$ such that $c < d$ and $F(c) < d$. Observe that if $(\alpha,\beta) \neq (\alpha',\beta')$ then the sets $\mathbf{S}_{c}(\alpha',\beta',X)$ and $\mathbf{S}_{c}(\alpha,\beta,X)$ are disjoint. Hence with $\varepsilon > 0$, when $X$ is sufficiently large, with Lemma \ref{betaalphaauxlemmasum2squares} and Lemma \ref{LowerboundScalphabetaX} we have the estimate
\begin{align*}
\# \mathcal{R}_{\theta,1}(d X^{1/2},X) &\geq \sum_{\substack{2 \leq \alpha \leq (c/4)X^{1/6} \\ \alpha \equiv 0 \pmod 2}}\left(\sum_{\substack{1\leq \beta \leq \alpha \\ \gcd(\alpha,\beta) = 1}}\#\mathbf{S}_{c}(\alpha,\beta,X)\right)\\
&\geq \sum_{\substack{2 \leq \alpha \leq (c/4)X^{1/6} \\ \alpha \equiv 0 \pmod 2}}\left(\sum_{\substack{1\leq \beta \leq \alpha \\ \gcd(\alpha,\beta) = 1}}\frac{c\cdot X^{1/6}}{4\alpha}\right)\\
&\gg \sum_{\substack{2 \leq \alpha \leq (c/4)X^{1/6} \\ \alpha \equiv 0 \pmod 2}}X^{1/6-\varepsilon}\\
&\gg X^{1/3-\varepsilon}.
\end{align*}
Thus, we prove that for sufficiently large $X$ we have
\begin{align*}
\#\mathcal{R}_{\theta,1}(X,X) \geq \#\mathcal{R}_{\theta,1}(dX^{1/2},X)\gg X^{1/3-o(1)},
\end{align*}
which proves Lemma \ref{Rtheta1example}.
\section{Proof of Lemma \ref{Rtheta2example}}
\label{Substitutioninspirationcubicpol}
We carry over useful formulas and lemmas from the previous section. We recall
\begin{align*}
g_{0}(\mathbf{u},\mathbf{v}) = &(u_{0}^2-2a_{0}u_{1}u_{2}+u_{2}^2a_{0}a_{2})+\\
&+(v_{0}^2-2a_{0}v_{1}v_{2}+v_{2}^2a_{0}a_{2}),
\end{align*}
\begin{align*}
g_{1}(\mathbf{u},\mathbf{v}) = &\left(2u_{0}u_{1}-2a_{1}u_{1}u_{2}+u_{2}^2(a_{1}a_{2}-a_{0})\right) +\\
&+\left(2v_{0}v_{1}-2a_{1}v_{1}v_{2}+v_{2}^2(a_{1}a_{2}-a_{0})\right),
\end{align*}
\begin{align*}
g_{2}(\mathbf{u},\mathbf{v}) = &(2u_{0}u_{2}+u_{1}^2 -2a_{2}u_{1}u_{2}+u_{2}^2(a_{2}^2-a_{1}))+\\
&+(2v_{0}v_{2}+v_{1}^2 -2a_{2}v_{1}v_{2}+v_{2}^2(a_{2}^2-a_{1})).
\end{align*}
\begin{align*}
I(u_{1},u_{2},v_{1},v_{2}) = u_{1}v_{2}-v_{1}u_{2}.
\end{align*}
We now set
\begin{align*}
s_{1}(u_{1},u_{2},& v_{1},v_{2}) = \\
&-2+2a_{1}u_{1}u_{2}-u_{2}^2(a_{1}a_{2}-a_{0})+2a_{1}v_{1}v_{2}-v_{2}^2(a_{1}a_{2}-a_{0}),
\end{align*}
\begin{align*}
s_{2}(u_{1},u_{2},& v_{1},v_{2}) = \\
& -u_{1}^2+2a_{2}u_{1}u_{2}-u_{2}^2(a_{2}^2-a_{1}) -v_{1}^2+2a_{2}v_{1}v_{2}-v_{2}^2(a_{2}^2-a_{1}),
\end{align*}
\begin{align*}
L_{0}(u_{1},u_{2},v_{1},v_{2}) = v_{2} s_{1}(u_{1},u_{2},v_{1},v_{2})/2 - v_{1} s_{2}(u_{1},u_{2},v_{1},v_{2})/2
\end{align*}
and
\begin{align*}
M_{0}(u_{1},u_{2},v_{1},v_{2}) = -u_{2} s_{1}(u_{1},u_{2},v_{1},v_{2})/2 +u_{1}s_{2}(u_{1},u_{2},v_{1},v_{2})/2.
\end{align*}
Let $\mathbf{T} \subset \Z[\theta]\times\Z[\theta]$ contain tuples $((u_{0}+u_{1}\theta+u_{2}\theta^2), (v_{0}+v_{1}\theta+v_{2}\theta^2))$ such that 
\begin{enumerate}
\item $v_{2}$ is even, $\gcd(u_{2},v_{2}) = 1$
\item $u_{1}$ and $v_{1}$ are odd, with $I(u_{1},u_{2},v_{1},v_{2}) = 1$
\item $u_{0} = L_{0}(u_{1},u_{2},v_{1},v_{2})$ and $v_{0} = M_{0}(u_{1},u_{2},v_{1},v_{2})$.
\end{enumerate}
We verify that on $\mathbf{T}$ we have $s_{1}(u_{1},u_{2},v_{1},v_{2}) \equiv 0 \mod 4$ and $s_{2}(u_{1},u_{2},v_{1},v_{2}) \equiv 2 \mod 4$. This implies that 
\begin{align}
\label{u0mod2lem1.6}
u_{0}= L_{0}(u_{1},u_{2},v_{1},v_{2}) \equiv 1 \mod 2
\end{align}
 and
\begin{align}
\label{v0mod2lem1.6}
v_{0} = M_{0}(u_{1},u_{2},v_{1},v_{2}) \equiv 1 \mod 2.
\end{align}
\begin{lem}
\label{TsubsetRtheta2inftyinfty}
We have $\mathbf{T} \subset \mathcal{R}_{\theta,2}(\infty,\infty)$.
\end{lem}
\begin{proof}
By Lemma \ref{Expansionubfvbfsum2squares} we have
\begin{align*}
\left((u_{0}+u_{1}\theta+u_{2}\theta^2),(v_{0}+v_{1}\theta+v_{2}\theta^2)\right) \in \mathcal{R}_{\theta,2}(\infty,\infty)
\end{align*}
if and only if all the listed conditions
\begin{itemize}
\item a) $g_{0}(\mathbf{u},\mathbf{v}) \equiv 0 \mod 2$
\item b) $g_{1}(\mathbf{u},\mathbf{v}) = -2$
\item c) $g_{2}(\mathbf{u},\mathbf{v}) = 0$
\end{itemize}
are satisfied. Solving the last two equations, we have
\begin{align*}
2u_{1}u_{0}+2v_{1}v_{0} = -2+2a_{1}u_{1}u_{2}-u_{2}^2(a_{1}a_{2}-a_{0})+2a_{1}v_{1}v_{2}-v_{2}^2(a_{1}a_{2}-a_{0}),
\end{align*}
and
\begin{align*}
2u_{2}u_{0}+2v_{2}v_{0} = -u_{1}^2+2a_{2}u_{1}u_{2}-u_{2}^2(a_{2}^2-a_{1}) -v_{1}^2+2a_{2}v_{1}v_{2}-v_{2}^2(a_{2}^2-a_{1}),
\end{align*}
where the right-hand sides of the above equations are $s_{1}(u_{1},u_{2},v_{1},v_{2})$ and $s_{2}(u_{1},u_{2},v_{1},v_{2})$ respectively. This leads us to the matrix equation
\begin{align*}
\begin{bmatrix}
2u_{1} & 2v_{1}\\
2u_{2} & 2v_{2}
\end{bmatrix}
\times
\begin{bmatrix}
u_{0} \\
v_{0}
\end{bmatrix}
=
\begin{bmatrix}
s_{1}(u_{1},u_{2},v_{1},v_{2})\\
s_{2}(u_{1},u_{2},v_{1},v_{2})
\end{bmatrix}
.
\end{align*}
If $I(u_{1},u_{2},v_{1},v_{2}) \neq 0$, solving the matrix equation gives us
\begin{align*}
\begin{bmatrix}
u_{0}\\
v_{0}
\end{bmatrix}
&= \frac{1}{u_{1}v_{2}-v_{1}u_{2}}
\begin{bmatrix}
v_{2} & -v_{1}\\
-u_{2} & u_{1}
\end{bmatrix}
\times
\begin{bmatrix}
s_{1}(u_{1},u_{2},v_{1},v_{2})/2\\
s_{2}(u_{1},u_{2},v_{1},v_{2})/2
\end{bmatrix}
\\
&= \frac{1}{u_{1}v_{2}-v_{1}u_{2}}
\begin{bmatrix}
v_{2} s_{1}(u_{1},u_{2},v_{1},v_{2})/2 - v_{1} s_{2}(u_{1},u_{2},v_{1},v_{2})/2 \\
-u_{2} s_{1}(u_{1},u_{2},v_{1},v_{2})/2 +u_{1}s_{2}(u_{1},u_{2},v_{1},v_{2})/2
\end{bmatrix}\\
&=\frac{1}{I(u_{1},u_{2},v_{1},v_{2})}
\begin{bmatrix}
L_{0}(u_{1},u_{2},v_{1},v_{2})\\
M_{0}(u_{1},u_{2},v_{1},v_{2}).
\end{bmatrix}
\end{align*}
We thus verify that elements of $\mathbf{T}$ satisfy conditions $b)$ and $c)$. Now, note that for the elements of $\mathbf{T}$ we have $g_{0}(\mathbf{u},\mathbf{v}) \equiv u_{0}^2+v_{0}^2 \mod 2$. By the modular equations \eqref{u0mod2lem1.6} and \eqref{v0mod2lem1.6}, condition $a)$ is verified, and thus the proof is complete.
\end{proof}
The rest of the proof now follows as in the discourse after the proof of Lemma \ref{SsubsetRtheta1inftyinfty}, but with $\mathbf{S}$ replaced with $\mathbf{T}$, $h_{i}$ replaced with $s_{i}$, $U_{0}$ replaced with $L_{0}$, $V_{0}$ replaced with $M_{0}$, $\mathcal{R}_{\theta,1}$ replaced with $\mathcal{R}_{\theta,2}$ and scaling absolute constants.
\section{Further Remarks}
\label{Remarkcubicpolys}
To obtain polynomial substitutions $R(t),S_{1}(t),S_{2}(t)$ described in \eqref{Substitutionforncubedmin2} that satisfy $R(t)^3-2 = S_{1}(t)^2 +S_{2}(t)^2$, we work with the system $\mathbf{T}$, with the added constraint that $v_{2} = 2$ and $u_{2} = 1$. In this subsystem the parameters $u_{0},v_{0},\ldots,u_{2},v_{2}$ are decided by $t$, where $u_{1} = 2t+1$. With these considerations, we let $R(t) = g_{0}(\mathbf{u},\mathbf{v})/2$, and we find $S_{1}(t), S_{2}(t)$ by managing the real and imaginary parts of the expansion
\begin{align*}
\prod_{j=0}^{2}\left(u_{0}+u_{1}2^{1/3}e^{2\pi i j/ 3}+u_{2}2^{2/3}e^{4 \pi i j/ 3} + i(v_{0}+v_{1}2^{1/3}e^{2\pi i j/ 3}+v_{2}2^{2/3}e^{4 \pi i j/ 3})\right).
\end{align*}
Theorem \ref{cubictheorem} concerns itself with $26.5625\%$ of degree-three monic irreducible polynomials with negative discriminant. This limitation arises because we only transferred lower bounds from $\#\mathcal{R}_{\theta,q}(\cdot,\cdot)$  to $\#B_{\mathbf{p}}(\cdot)$, where $q \in\{1,2\}$. It is important to note that sometimes $\mathcal{R}_{\theta,q}(\cdot,\cdot)$ can be empty, for example, if $q = 1$ and $a_{1}a_{2}-a_{0}$ is even. To cover all cases of irreducible monic polynomials of degree-three with negative discriminant, we expect one would need to use sets $\mathcal{R}_{\theta,2^{t}}(\cdot,\cdot)$ for a suitable integer $t$ and employ enhanced parity arguments; however, we do not pursue these additional complications here. Furthermore, our arguments should generalise to study integer points on the surface $y^2+n z^2 = \mathbf{p}(x)$. Here, our theory is stronger when $n \in \mathbb{N}$ and $\mathbf{p}$ is a monic cubic polynomial with a negative discriminant, as our unit arguments can be transferred quite readily. However, when $n$ is negative, most of our theory breaks down (with the exception of an analogue for Lemma \ref{Ringtransferprinclemsum2squares}). This is because our arguments, even after suitable modification, do not provide sufficient information about the density of units in the ring $\mathcal{O}_{\Q(\theta,\sqrt{-n})}$.
\section{Acknowledgements}
The author expresses their gratitude to Igor Shparlinski for pointing out a relevant result by Kronecker: if a monic irreducible polynomial has all its roots on the unit circle, it is a cyclotomic polynomial. The author also thanks an anonymous referee for their careful reading and for identifying a flaw in a previous iteration of the manuscript. This research was generously supported by the Australian Government Research Training Program (RTP) Scholarship and a top-up scholarship from the University of New South Wales.
\section{Data availability statement}
Data sharing is not applicable to this article as no datasets were generated or analysed during the current study.

\end{document}